\documentclass[12pt,a4]{article}

\usepackage{epsfig}
\usepackage{amsmath}
\usepackage{amsfonts}
\usepackage{amsthm}

        \theoremstyle{plain}
        \newtheorem{lemma}{Lemma}[section]
        \newtheorem{proposition}{Proposition}[section]
        \newtheorem{theo}{Theorem}[section]

        \theoremstyle{remark}
        \newtheorem{remark}{\bf Remark}[section]
        \theoremstyle{remark}
        
        \theoremstyle{remark}
        
\newcommand{\demi}{\frac{1}{2}}

\newcommand{\R}{{\mathbb{R}}}
\newcommand{\Z}{{\mathbb{Z}}}

\newcommand{\dt}{\partial_t}
\newcommand{\dx}{\partial_x}

\newcommand{\eps}{\varepsilon}

\newcommand{\av}[1]{\left[ #1 \right]}

\newcommand{\Dx}{\Delta x}
\newcommand{\Dt}{\Delta t}

\newcommand{\tn}{t_n}
\newcommand{\tnp}{t_{n+1}}
\newcommand{\x}{x_i}
\newcommand{\xp}{x_{i+1}}
\newcommand{\xpd}{x_{i+\demi}}

\newcommand{\half}{\frac{1}{2}}
\newcommand{\thalf}{\frac{3}{2}}

\newcommand{\Dm}{D^-}
\newcommand{\Dp}{D^+}
\newcommand{\Dc}{D^c}
\newcommand{\Dnot}{D^0}
\newcommand{\dnot}{\delta^0}
\newcommand{\sumi}{\sum_{i\in\Z}}

\newcommand{\rhon}{\rho^n}
\newcommand{\rhoni}{\rho^n_{i}}
\newcommand{\rhonim}{\rho^n_{i-1}}
\newcommand{\rhonip}{\rho^n_{i+1}}

\newcommand{\rhonp}{\rho^{n+1}}
\newcommand{\rhonpi}{\rho^{n+1}_{i}}

\newcommand{\gn}{g^n}
\newcommand{\gni}{g^n_{i+\half}}
\newcommand{\gnim}{g^n_{i-\half}}
\newcommand{\gnip}{g^n_{i+\frac{3}{2}}}

\newcommand{\gnp}{g^{n+1}}
\newcommand{\gnpi}{g^{n+1}_{i+\half}}
\newcommand{\gnpim}{g^{n+1}_{i-\half}}

\newcommand{\trhon}{\tilde{\rho}^n}
\newcommand{\trhoni}{\tilde{\rho}^n_{i}}

\newcommand{\trhonp}{\tilde{\rho}^{n+1}}
\newcommand{\trhonpi}{\tilde{\rho}^{n+1}_{i}}

\newcommand{\tgn}{\tilde{g}^n}
\newcommand{\tgni}{\tilde{g}^n_{i+\half}}

\newcommand{\tgnp}{\tilde{g}^{n+1}}
\newcommand{\tgnpi}{\tilde{g}^{n+1}_{i+\half}}

\newcommand{\anpi}{a^{n}_i}
\newcommand{\ani}{a^{n}_i}
\newcommand{\an}{a^{n}}
\newcommand{\anp}{a^{n}}

\newcommand{\bni}{b^{n}_i}
\newcommand{\bn}{b^{n}}

\newcommand{\rnorm}[1]{\|#1\|}
\newcommand{\gnorm}[1]{\||#1\||}
\newcommand{\innerp}[2]{\left\langle#1 \, , \, #2\right\rangle}

\newcommand{\tsigma}{\tilde{\sigma}}

\begin{document}

\begin{flushright}
{\it draft, for submission to SINUM, \today}
\end{flushright}

\begin{center}

{\bf Analysis of an asymptotic preserving scheme for linear kinetic
  equations in the diffusion limit\\}

\vspace{1cm}

Jian-Guo Liu\footnote{Department of Physics and Department of Mathematics, Duke University,
Durham, NC 27708, USA  ({\bf \tt  Jian-Guo.Liu@duke.edu})},
Luc Mieussens\footnote{Institut de Math\'ematiques de Bordeaux (UMR
  5251), Universit\'e de Bordeaux, 351, cours de la Lib\'eration, 33405 Talence cedex, France ({\bf
  \tt Luc.Mieussens@math.u-bordeaux1.fr})}

\end{center}

\bigskip

{\bf Abstract. } We present a mathematical analysis of the asymptotic preserving scheme
proposed in [M. Lemou and L. Mieussens, {\em SIAM J. Sci. Comput.},
31, pp. 334--368, 2008] for linear transport equations in kinetic and
diffusive regimes. We prove that
the scheme is uniformly stable and accurate with respect to the mean free path of
the particles. This property is satisfied under an explicitly given
CFL condition. This condition tends to a parabolic CFL condition for
small mean free paths, and is close to a convection CFL condition for large mean free paths. Our analysis is
based on very simple energy estimates.



                        \section{Introduction}

Particle systems are often described at the
microscopic level by kinetic models (neutron
transport, radiative transfer, electrons in semi-conductors, or rarefied gas
dynamics). Simulating such systems by using kinetic models can be computationally very
expensive, but modern super computers now enable realistic
simulations. When the mean free path of the particles is very small as
compared to the (macroscopic) size of the computational domain, the kinetic model
can be very well approximated by a much simpler macroscopic model
(diffusion equation, Rosseland approximation, Euler and Navier-Stokes equations), that can be numerically solved
much faster.

However, there are many cases where the ratio mean free
path/macroscopic size (the so-called ``Knudsen number'' in rarefied
gas dynamics, denoted by $\eps$ in this paper) is not constant: depending on the geometry of
the boundaries, or on the boundary conditions, this ratio may vary in
time, and in space. In such multiscale situations, usual kinetic
solvers are often useless: for stability and accuracy reasons, they
must resolve the microscopic scales, which is computationally to
expensive in ``fluid'' zones (where $\eps$ is small). By contrast,
macroscopic solvers are faster but may be inaccurate in ``kinetic''
zones (where $\eps$ is large).

This is why many people have been working for more than 20 years on a
kind of multiscale kinetic schemes: the asymptotic-preserving (AP)
schemes. Such schemes are uniformly stable with respect to $\eps$
(thus their computational complexity does not depend on $\eps$), and
are consistent with the macroscopic model when $\eps$ goes to $0$ (the
limit of the scheme is a scheme for the macroscopic model). 

Up to our knowledge, AP schemes have first been studied (for steady
problems) in neutron transport by Larsen, Morel and
Miller~\cite{LMM1987}, Larsen and Morel~\cite{LM1989}, and then by Jin
and Levermore~\cite{JL1991,JL1993}. For unstationary problems, the
difficulty is the time stiffness due to the collision operator. To
avoid the use of expensive fully implicit schemes, two
classes of semi-implicit time discretizations have been proposed by
Klar~\cite{klar_sinum_1998} and Jin, Pareschi and
Toscani~\cite{JPT2000} (see preliminary works
in~\cite{JPT1998,jin_sisc_1999} and extensions
in~\cite{JP2000b,JP2000,NP1998,klar_sinum_1999,klar_sisc_1999}. Similar
ideas have also been used by Buet et al. in~\cite{BCLM}.

In~\cite{LM2007}, Lemou and Mieussens have proposed a new AP scheme
based on the micro-macro decomposition of the distribution function into microscopic and
macroscopic components (similar schemes have also been
proposed by Klar and
Schmeiser~\cite{KS2001}, and more recently by Carrillo et al.~\cite{CGL2008,CGLV2008}). A coupled
system of equations is obtained for these two components without any
linearity assumption. The decomposition only uses basic properties of
the collision operator that are common to most of kinetic equations
(namely conservation and equilibrium properties). Then this system is
solved with a suitable time semi-implicit discretization and space finite
differences on staggered grids. While almost all the schemes
mentioned before are based on very similar ideas, the approach
proposed in~\cite{LM2007} has been shown to be very general, since
it applies to kinetic equations for both diffusion and
hydrodynamic regimes (for the diffusion regime, see~\cite{LM2007} for linear
transport equations and~\cite{BLM2009} for the non-linear Kac
equation, for the hydrodynamic regime, see~\cite{BLM2008} for the
Boltzmann equation). We also mention the work of Degond, Liu and
Mieussens~\cite{DLM2006} who proposed a similar approach (micro-macro
decomposition) to design macroscopic diffusion models with kinetic
upscalings: this approach also leads to AP schemes, at least for a
semi-discrete time discretization (no AP space discretization was
studied in this paper). 

While many different schemes have been proposed in the past few
years, it appears that the rigorous proof of their AP property looks
rather difficult and is seldom investigated. Up to our knowledge,
there are only two papers on this subject. Klar and
Unterreiter~\cite{KU2002} have proved that a scheme similar to that
of~\cite{klar_sinum_1998} and~\cite{JPT2000} for the linear transport
equation is uniformly stable. However, their proof is based on a
von Neumann analysis, and hence is restricted to a one
dimensional equation with constant coefficients in a periodic
domain. Gosse and Toscani have proposed in~\cite{GT2004} an AP
scheme based on a rather different idea (discretization of the
collision term as a non-conservative product and use of well-balanced Godunov
schemes): they have been able to prove a uniform stability
property, still in the linear case, but also a strong positivity property. However, their scheme is
based on techniques (like approximation of the steady solution) that are difficult to generalize to other equations.

In this paper, we propose a very simple stability proof for the AP
scheme of~\cite{LM2007} and we exhibit an explicit CFL condition. This
condition is uniform with respect to $\eps$, and gives a standard
parabolic CFL condition $\Dt=O(\Dx^2)$ when $\eps$ goes to 0.  While a
stability property was already proved in~\cite{LM2007} for this
scheme, this was only for a simple two-velocity model (telegraph
equation), by using von Neumann analysis. Here, our proof is based on
energy estimates that are more general than von Neumann analysis, and
hence is valid for one-dimensional linear equations with non-constant
coefficients, continuous velocity variable, in the whole space. By the
same technique, we are also able to prove uniform error estimates. 

Our paper is organized as follows. In section~\ref{sec:scheme}, we
introduce a general linear equation, we present its discretization by
the AP scheme, and we give the main features of this scheme. Then, in
section~\ref{sec:stab}, we give our main stability result and its
proof. The error estimates are given in section~\ref{sec:acc}.

                        \section{An AP scheme for the linear transport equation}
\label{sec:scheme}

Linear transport equation is a model for the evolution of particles in
some medium (neutron transport, linear radiative transfer,
\ldots). Generally, this model reads, in scaled variables 
\begin{equation*}
\eps \dt \phi + \Omega\cdot  \nabla_{r} \phi =
\frac{\sigma}{\eps}L\phi -\eps\sigma_A  \phi + \eps S,  
\end{equation*}
where $\phi(t,r,\Omega)$ is the number density of particles in the
position-direction phase space that depends on time $t$, position $r=(x,y,z)\in \R^3$,
and angular direction of propagation of particles $\Omega\in S^2$.
Moreover, $\sigma$ is the total cross section, $\sigma_A$ is the
absorption cross section, and $S$ is an internal source of particles,
which is independent of $\Omega$. The linear operator $L$ models the
scattering of the particles by the medium and acts only on the angular
dependence of $\phi$. This simple model does not allow for particles
of possibly different energy (or frequency); it is called
``one-group'' or ``monoenergetic'' equation. The parameter $\eps$ is a
scale factor that measures the ratio between a typical microscopic
length (the mean free path of a particle, for instance) to a typical
macroscopic length (the size of the computational domain, for
instance). See~\cite{LMM1987} for details.

In this paper, we consider this one-group equation in the slab
geometry: we assume that $\phi$  depends only on the slab axis
variable $x\in \R$. Then it can be shown that the average of $\phi$
with respect to the $(y,z)$ cosine directions of $\Omega$, denoted by
$f(t,x,v)$, satisfies the one-dimensional equation
\begin{equation}  \label{eq-onegroup}
\eps \dt f + v \dx f = \frac{\sigma}{\eps}Lf -\eps\sigma_A  f + \eps
S,
\end{equation}
where $v\in[-1,1]$ is the $x$ cosine direction of $\Omega$. At $t=0$, we have
the initial data $f(0,x,v)=f^0(x,v)$. We assume that the
cross sections satisfy the inequalities $0<\sigma_m\leq \sigma(x)\leq
\sigma_M$ and $0\leq \sigma_A(x)\leq \sigma_{AM}$ for every $x$. We do
not consider boundary conditions in this paper, and hence $x\in\R$.

The linear operator
$L$ is given by
\begin{equation}  \label{eq-defL}
Lf(v)=\int_{-1}^1s(v,v')(f(v')-f(v))\, dv',
\end{equation}
where the kernel $s$ is such that $0<s_m\leq s(v,v')\leq s_M$ for
every $v,v'$ in $[-1,1]$. We also assume that $s$ satisfies $\int_{-1}^1s(v,v')\,dv'=1$, and
that it is symmetric: $s(v,v')=s(v,v')$.
For the sequel, it is useful to define the operator $\av{.}$ such
that $\av{\phi}=\half\int_{-1}^1\phi(v)\, dv$ is the average of every
velocity dependent function $\phi$. With these assumptions, it is standard~\cite{BSS} to state the
following properties of $L$: 
\begin{proposition} \label{prop:L}
  \begin{itemize}
 \item $\av{L\phi}=0$ for every $\phi$ in $L^2([-1,1])$
  \item the null space of $L$ is ${\cal N}(L)=\lbrace \phi
    =\av{\phi}\rbrace$ (constant functions)
  \item the rank of $L$ is ${\cal R}(L)={\cal N}^\perp(L)=\lbrace \phi
    \text{ s.t } \av{\phi}=0\rbrace$
  \item $L$ is non-positive self-adjoint in $L^2([-1,1])$ and we have
\begin{equation}  \label{eq-sadjoint}
\av{\phi L\phi}\leq -2 s_m \av{\phi^2}
\end{equation}
for every $\phi\in{\cal N}^\perp(L)$
\item $L$ admits a pseudo-inverse from ${\cal N}^\perp(L)$ onto
  ${\cal N}^\perp(L)$, denoted by $L^{-1}$
\item the orthogonal projection from $L^2([-1,1])$ onto ${\cal
    N}^\perp(L)$ is $\av{.}$
  \end{itemize}
\end{proposition}

When $\eps$ becomes small (the ``diffusion'' regime), it is well known that the solution $f$
of~(\ref{eq-onegroup}) tends to its own
average density $\rho=\av{f}$, which is a solution of the asymptotic
diffusion limit
\begin{equation}  \label{eq-diff}
\dt \rho -\dx \kappa  \dx\rho =-\sigma_A\rho +S,
\end{equation}
where the diffusion coefficient is
$\kappa(x)=-\frac{\av{vL^{-1}v}}{\sigma(x)}$. An asymptotic preserving
scheme for the linear kinetic equation~(\ref{eq-onegroup}) is a
numerical scheme that discretizes~(\ref{eq-onegroup}) in such a way
that it leads to a correct discretization of the diffusion
limit~(\ref{eq-diff}) when $\eps$ is small.

Now, we summarize the results obtained in~\cite{LM2007}. By using the
micro-macro decomposition $f=\rho+\eps g$, where $\rho=\av{f}$ and $g$
is such that $\av{g}=0$, we
derived the micro-macro model for~(\ref{eq-onegroup}) that reads
\begin{subequations} \label{eq-mM} 
\begin{align}
& \dt \rho + \dx\av{vg} = -\sigma_A\rho +S ,\label{eq-mMrho}  \\
& \dt g + \frac{1}{\eps} (I-\av{.}) (v\dx g) =
-\frac{\sigma}{\eps^2}Lg 
- \frac{1}{\eps^2}v \dx\rho,\label{eq-mMg} 
\end{align}
\end{subequations}
with initial data $\rho^0=\av{f^0}$ and $\eps g^0=f^0-\rho^0$. 
This system if formally equivalent to~(\ref{eq-onegroup}). 

Then, we
proposed the following numerical scheme for this system. We choose a
time step $\Dt$ and times $t_n=n\Dt$, and two staggered grids of step
$\Dx$ and nodes $x_i=i\Dx$ and $x_{i+\half}=(i+\half)\Dx.$ We use the
approximated values $\rhoni\approx\rho(t_n,x_i)$ and
$\gni(v)\approx g(t_n,x_{i+\half},v)$, and the scheme reads
\begin{subequations} \label{eq-scheme} 
\begin{align}
& \frac{\rhonpi-\rhoni}{\Dt} +
\av{v\frac{\gnpi-\gnpim}{\Dx}} = -\sigma_{A,i}\rhonpi+S_i,
 \label{eq-schrho}  \\
& \frac{\gnpi-\gni}{\Dt} + \frac{1}{\eps\Dx} (I-\av{.})
\left(v^+(\gni-\gnim)+v^-(\gnip-\gni)\right)
\label{eq-schg}  \\
& \qquad\qquad =
- \frac{\sigma_{i+\half}}{\eps^2} L\gnpi -
\frac{1}{\eps^2}v\frac{\rhonip-\rhoni}{\Dx}, \nonumber
\end{align} 
\end{subequations}
where $v^\pm=\frac{v\pm|v|}{2}$. Note that as in the continuous case,
this scheme preserves the zero average of $g$: 
\begin{proposition}
If the initial data $g^0$ satisfies $\av{g^0_{i+\half}}=0$ for every $i$, then for every $n$ and $i$, we have
\begin{equation}  \label{eq-avg}
\av{\gni}=0.
\end{equation}
\end{proposition}
\begin{proof}
  Apply the average operator $\av{.}$ to~(\ref{eq-schg}): since we
  have $\av{I-\av{.}}=0$ (obvious), $\av{L}=0$
  (proposition~\ref{prop:L}), and $\av{v}=0$ (obvious), this relation
  yields $\av{\gnpi}-\av{\gni}=0$, which gives the result.
\end{proof}

In scheme~(\ref{eq-scheme}), the upwind discretization of
$(I-\av{.}) (v\dx g)$ is to insure stability in the kinetic
regime, while the centered approximations of $\dx\av{vg}$ and $v
\dx\rho$ are to capture the diffusion limit. Indeed, it is clear
that when $\eps$ goes to 0, we have from~(\ref{eq-schg})
$\gnpi=-\frac{1}{\sigma_{i+\half}}L^{-1}\left(v\frac{\rhonip-\rhoni}{\Dx}\right)+O(\eps)$.
Consequently, the flux of $\gnpi$ is
\begin{equation*}
  \av{v\gnpi}=-\frac{1}{\sigma_{i+\half}}\av{vL^{-1}v}\frac{\rhonip-\rhoni}{\Dx}+O(\eps),
\end{equation*}
where we have used the facts that $\rhoni$ does not depend on $v$ and
that $L$---and hence $L^{-1}$ too---only applies to functions of $v$.
Then, using this relation in~(\ref{eq-schrho}), we get
\begin{equation}  \label{eq-sdiff}
\frac{\rhonpi-\rhoni}{\Dt} -\frac{1}{\Dx}\left(
  \kappa_{i+\half}\frac{\rhonip-\rhoni}{\Dx} 
-   \kappa_{i-\half}\frac{\rhoni-\rhonim}{\Dx} 
  \right)
=-\sigma_{A,i}\rhoni + S_i,
\end{equation}
with $\kappa_{i+\half}=-\frac{\av{vL^{-1}v}}{\sigma_{i+\half}}$, which
is the usual 3-points stencil explicit scheme for the diffusion
equation~(\ref{eq-diff}).

Of course, this property is true only if
$\Dt$ can be chosen independently of $\eps$, or in other words, if the
scheme is uniformly stable with respect to $\eps$. In~\cite{LM2007},
we have proved that this scheme is indeed uniformly stable under some
CFL condition, in the simpler case of the telegraph equation (in which
we have only two discrete velocities $v=\pm 1$, and $s\equiv1$,
$\sigma\equiv 1$, $\sigma_A=S\equiv 0$). 
In the following section, we extend this result to the general
equation~(\ref{eq-onegroup}) for scheme~(\ref{eq-scheme}).

\section{Uniform stability}
\label{sec:stab}

We give our main result of stability for
scheme~(\ref{eq-scheme}). Without lost of generality, we assume $S=0$
(source free case).
\begin{theo} \label{theo:stab}
If $\Dt$ satisfies the following CFL condition
\begin{equation}  \label{eq-CFL}
\Dt\leq  \frac{1}{3}\left(\tsigma\Dx^2+2\eps\Dx\right),
\end{equation}
with $\tsigma=2s_m\sigma_m$, then the sequences $\rhon$ and
$\gn$ defined by scheme~(\ref{eq-scheme}) satisfy the energy estimate
\begin{equation*}
\sumi\left(\rhoni\right)^2\Dx + \eps^2 \sumi\av{\left(\gni\right)^2}\Dx \leq\sumi\left(\rho^0_i\right)^2\Dx + \eps^2 \sumi\av{\left(g^0_{i+\half}\right)^2}\Dx
\end{equation*}
for every $n$, and hence the scheme~(\ref{eq-scheme}) is stable.
\end{theo}

Note that CFL condition~(\ref{eq-CFL}) can be viewed as an average
of a diffusive CFL condition $\Dt\leq\tsigma\Dx^2$ (needed for the diffusion
scheme~(\ref{eq-sdiff})) and of a convection CFL $\Dt\leq \eps\Dx$. It shows that the scheme is stable
uniformly in $\eps$, that is to say a diffusive CFL condition
$\Dt\leq C\Dx^2$ is sufficient for stability for small
$\eps$, while a convection CFL is sufficient for
$\eps=O(1)$. 

Moreover, while our CFL condition~(\ref{eq-CFL}) is valid for every
$\eps$, we point out that it is not optimal for each $\eps$. This
might be the price to be paid for getting a uniform condition. For
instance, in the simpler case of constant cross sections
($\sigma\equiv\sigma_m$)  and kernel ($s=s_m\equiv
\half$), the diffusion coefficient of the diffusion
equation~(\ref{eq-sdiff}) is $\kappa=\frac{1}{3\sigma}$, and the
optimal CFL condition for scheme~(\ref{eq-sdiff}) is $\Dt\leq
\frac{\Dx^2}{2\kappa}$. However, CFL condition~(\ref{eq-CFL}) for our
AP scheme reads (for $\eps=0$) $\Dt\leq
\frac{2}{9}\frac{\Dx^2}{2\kappa}$, which is $0.2$ as small as the
optimal CFL.

\begin{remark}
  If we have a source term $S\neq 0$, then the same CFL condition naturally
  gives a linear growth of the energy. Since this is standard and
  does not lead to any additional difficulty, we do not consider this
  case here.
\end{remark}

\begin{remark} \label{rem:CFLabsorption}
  In the case of a time explicit discretization of the absorption term (that is to say when
$-\sigma_{A,i}\rhonpi$ is replaced by $-\sigma_{A,i}\rhoni$
in~(\ref{eq-scheme}), we have the same result with a CFL condition
which is a
bit more complicated. Namely, the scheme is stable if
\begin{equation}  \label{eq-CFLabsoption}
\Dt\leq \min \left( \frac{2}{1+\sigma_{A,M}},\frac{3}{3+\sigma_{A,M}}\Dt_S  \right),
\end{equation}
where $\Dt_S$ is the maximum time step allowed by CFL
condition~(\ref{eq-CFL}) for the scheme with implicit (or zero)
absorption term.
\end{remark}

This theorem is proved in the following sections: section~\ref{subsec:lemma} contains compact
notations and useful lemma, and section~\ref{subsec:ener} contains the derivation of our energy estimate.

   \subsection{Notations and useful lemma}
   \label{subsec:lemma}

We give some useful notations for norms and inner products that are
used in our analysis. For every grid function $\mu=(\mu_i)_{i\in\Z}$
we define: 
\begin{equation}  \label{eq-rnorm}
\rnorm{\mu}^2=\sumi\mu_i^2\Dx.
\end{equation}
For every velocity dependent grid function
$v\in[-1,1]\mapsto\phi(v)=(\phi_{i+\half}(v))_{i\in\Z}$, we define:
\begin{equation}  \label{eq-gnorm}
\gnorm{\phi}=\sumi\av{\phi_{i+\half}^2}\Dx.
\end{equation}
If $\phi$ and $\psi$ are two velocity dependent grid functions, we
define their inner product:
\begin{equation}  \label{eq-inner}
\innerp{\phi}{\psi}=\sumi\av{\phi_{i+\half}\psi_{i+\half}}\Dx.
\end{equation}

Now we give some notations for the finite difference operators that
are used in scheme~(\ref{eq-scheme}). For every grid function
$\phi=(\phi_{i+\half})_{i\in\Z}$, we define the following one-sided operators: 
\begin{equation} \label{eq-diff_onesided}
 \Dm \phi_{i+\half} = \frac{\phi_{i+\half}-\phi_{i-\half}}{\Dx} 
\quad \text{ and }\quad  \Dp \phi_{i+\half} = \frac{\phi_{i+\thalf}-\phi_{i+\half}}{\Dx}
\end{equation}
We also define the following centered operators:
\begin{equation} \label{eq-diff_centered}
\Dc \phi_{i+\half} = \frac{\phi_{i+\thalf}-\phi_{i-\half}}{2\Dx}
\quad \text{ and }\quad  \Dnot \phi_i=\frac{\phi_{i+\half}-\phi_{i-\half}}{\Dx} (=\Dm\phi_{i+\half}).
\end{equation}
Finally, for every grid function $\mu=(\mu_{i})_{i\in\Z}$, we define
the following centered operator:
\begin{equation} \label{eq-diff_centered2}
 \dnot \mu_{i+\half}=\frac{\mu_{i+1}-\mu_{i}}{\Dx}.
\end{equation}

\begin{lemma}[Centered form of the upwind operator] \label{lemma:5}
  For every grid function
$\phi=(\phi_{i+\half})_{i\in\Z}$, we have:
\begin{equation*}
  \left(v^+\Dm + v^-\Dp \right)\phi_{i+\half}=v\Dc\phi_{i+\half} -\frac{\Dx}{2}|v|\Dm\Dp\phi_{i+\half}.
\end{equation*}
\end{lemma}
\begin{proof}
  This result is easily obtained by using the relations
  $v^\pm=\frac{v\pm|v|}{2}$ and the identity $\Dm+\Dp=2\Dc$.
\end{proof}

\begin{lemma} \label{lemma:estimDp}
  For every grid function $\phi=(\phi_{i+\half})_{i\in\Z}$, we have:
\begin{equation*}
\sumi{\left(\Dp\phi_{i+\half}\right)^2}\Dx \leq \frac{4}{\Dx^2}\sum_i\phi_{i+\half}^2\Dx.  
\end{equation*}
\end{lemma}
\begin{proof}
  Expand the left-hand side, use the inequality $(a+b)^2\leq 2(a^2+b^2)$, and then
  use the change of index $i+\frac{3}{2}\rightarrow i+\half$.
\end{proof}

\begin{lemma}[Estimate for the adjoint upwind operator] \label{lemma:3}
  For every positive real number $\alpha$ and for every velocity dependent grid functions
$\phi$ and $\psi$, we have:
\begin{equation*}
\left|\innerp{\left(v^+\Dp + v^-\Dm \right)\psi}{\phi}\right| \leq \alpha
\gnorm{\phi}^2 + \frac{1}{4\alpha}\gnorm{|v|\Dp\psi}^2.  
\end{equation*}
\end{lemma}
\begin{proof}
  From Young's inequality, we obtain for any positive real number $\alpha$:
\begin{equation}\label{eq-young} 
  \left|\innerp{\left(v^+\Dp + v^-\Dm \right)\psi}{\phi}\right| \leq \alpha
\gnorm{\phi}^2 + \frac{1}{4\alpha}\gnorm{\left(v^+\Dp + v^-\Dm \right)\psi}^2.
\end{equation}
Now, the second term of the right-hand side of this inequality can be written:
\begin{equation*}
  \frac{1}{4\alpha}\gnorm{\left(v^+\Dp + v^-\Dm \right)\psi}^2 
    = \frac{1}{4\alpha}\sumi
\half \left( 
   \int_{-1}^0 v^2(\Dm\psi_{i+\half})^2\, dv 
 + \int_{0}^1  v^2(\Dp\psi_{i+\half})^2\, dv 
\right)\Dx. 
\end{equation*}
Then, with a simple changing of index, $\Dm$ can be replaced
by $\Dp$ in the first integral, which gives:
\begin{equation*}
\begin{split}
  \frac{1}{4\alpha}\gnorm{\left(v^+\Dp + v^-\Dm \right)\psi}^2 
&    =  \frac{1}{4\alpha}\sumi
\half \int_{-1}^1 v^2(\Dp\psi_{i+\half})^2\, dv\Dx  \\
& = \frac{1}{4\alpha}\gnorm{|v|\Dp\psi}^2 .
\end{split}
\end{equation*} 
Finally, using this inequality in~(\ref{eq-young}) gives the result.
\end{proof}

\begin{lemma}[Discrete integration by parts] \label{lemma:intbp}
  For every grid functions
$\phi=(\phi_{i+\half})_{i\in\Z}$, $\psi=(\psi_{i+\half})_{i\in\Z}$,
and $\mu=(\mu_{i})_{i\in\Z}$, we have:
\begin{align*}
&  \sumi \mu_i\Dnot\phi_i\Dx  = -\sumi  \bigl(\dnot\mu_{i+\half}\bigr)\phi_{i+\half}\Dx, \\
&  \sumi \psi_{i+\half}\Dm\phi_{i+\half}\Dx  = -\sumi  \bigl(\Dp\psi_{i+\half}\bigr)\phi_{i+\half}\Dx, \\
&  \sumi \phi_{i+\half}\Dc\phi_{i+\half}\Dx  = 0.
\end{align*}
\end{lemma}
\begin{proof}
These results are simply obtained by using obvious changing of indexes
and the definitions of the finite difference operators given in~(\ref{eq-diff_onesided}--\ref{eq-diff_centered2}).
\end{proof}

\begin{lemma} \label{lemma:2}
  If $g\in L^2([-1,1])$ then
\begin{equation*}
  \av{vg}^2\leq \half \av{|v|g^2}
\end{equation*}
\end{lemma}
\begin{proof}
We just note that
\begin{equation*}
\begin{split}
  \av{vg}^2& = \frac{1}{4}\left(\int_{-1}^1vg \, dv\right)^2  =
  \frac{1}{4}\left(\int_{-1}^1 {\rm sign}(v)\sqrt{|v|}\sqrt{|v|}g \,
    dv\right)^2 \\
& \leq \frac{1}{4} \int_{-1}^1 \left({\rm sign}(v)\sqrt{|v|}\right)^2 \, dv 
                 \int_{-1}^1 \left(\sqrt{|v|}g\right)^2 \, dv \qquad \text{ by
                   Cauchy-Schwarz inequality} \\
& = \half \av{|v|g^2}.
\end{split}
\end{equation*}
\end{proof}

   \subsection{Energy estimates}
   \label{subsec:ener}
Since the time discretization of absorption term $-\sigma_A\rho$ is
implicit, it plays no role in the energy estimate. Consequently, to
simplify the proof, we consider the case where there is no absorption
(see remark~\ref{rem:absorption} at the end of this section). We proceed in five short steps.

\noindent {\it Step 1.} \\
Here we derive a first energy relation. With the finite difference operators defined
in~(\ref{eq-diff_onesided})--(\ref{eq-diff_centered2}), the scheme can
be written in the following compact form:
\begin{subequations}\label{eq-cscheme} 
  \begin{align}
& \frac{\rhonpi-\rhoni}{\Dt} + \Dnot\av{v\gnp_i} = 0 \label{eq-srho}
\\
& \frac{\gnpi-\gni}{\Dt} + \frac{1}{\eps} (I-\av{.})
\left(v^+\Dm+v^-\Dp\right)\gni = \frac{\sigma_{i+\half}}{\eps^2} L\gnpi 
- \frac{1}{\eps^2}v\,\dnot \rhon_{i+\half}.\label{eq-sg} 
  \end{align}
\end{subequations}

We define the energy of system~(\ref{eq-mM}) as $\int_\R \rho^2 \, dx +
\eps^2 \int \av{g^2} \, dx$. It is clear that the scheme can be proved
to be stable if the energy at time $n+1$ can be controlled by the
energy at time $n$. Consequently, we first multiply~(\ref{eq-srho}) by
$\rhonpi$, then we take the sum over $i\in\Z$, and finally, we use
the standard equality $a(a-b)=\half(a^2-b^2+|a-b|^2)$ to get: 
\begin{subequations}
\begin{equation}  \label{eq-Erho}
\frac{1}{2\Dt} \left(  \rnorm{\rhonp}^2-\rnorm{\rhon}^2+\rnorm{\rhonp-\rhon}^2\right)
+ \sumi \rhonpi\Dnot\av{v\gnp_i}\Dx = 0.
\end{equation}
Second, we multiply~(\ref{eq-sg}) by $\gnpi$, we take the velocity
average, we sum over $i\in \Z$, and we get:
\begin{equation}  \label{eq-Eg}
\begin{split}
& \frac{1}{2\Dt} \left( \gnorm{\gnp}^2-\gnorm{\gn}^2 + \gnorm{\gnp-\gn}^2  \right)
+ \frac{1}{\eps}\innerp{\gnp}{(I-\av{.})\left(v^+\Dm+v^-\Dp\right)\gn} \\
&  =  \frac{1}{\eps^2} \innerp{\gnp}{\sigma L\gnp} - \frac{1}{\eps^2}\sumi \av{v\gnpi}\dnot \rhon_{i+\half}\Dx.
\end{split}
\end{equation}
\end{subequations}
Now we use relation~(\ref{eq-avg}): since $\av{\gnpi}=0$ for every
$i$, a simple expansion of the inner product of the left-hand side of~(\ref{eq-Eg}) shows
that it can be reduced to:
\begin{equation}  \label{eq-Innp}
\begin{split}
& \innerp{\gnp}{(I-\av{.})\left(v^+\Dm+v^-\Dp\right)\gn}   \\
& =  \innerp{\gnp}{\left(v^+\Dm+v^-\Dp\right)\gn} -  \sumi \av{\gnpi}\av{\left(v^+\Dm+v^-\Dp\right)\gni}\Dx \\
& = \innerp{\gnp}{\left(v^+\Dm+v^-\Dp\right)\gn}.
\end{split}
\end{equation}
Moreover, we can use relation~(\ref{eq-sadjoint}) and the assumptions
of the cross section and the kernel to estimate the inner product of
the right-hand side of~(\ref{eq-Eg}) as follows:
\begin{equation}\label{eq-gsLg}
\begin{split}
\innerp{\gnp}{\sigma L\gnp}  & = \sumi \sigma_{i+\half}\av{\gnpi L\gnpi}\Dx \\
& \leq -\underbrace{2s_m\sigma_m}_{\tsigma} \gnorm{\gnp}^2 .
\end{split}
\end{equation}

Consequently, we add up~(\ref{eq-Erho}) and $\eps^2$
times~(\ref{eq-Eg}), then we use~(\ref{eq-Innp}),~(\ref{eq-gsLg}), and
the discrete integration by parts of lemma~\ref{lemma:intbp} to get our preliminary energy estimate:
\begin{equation}  \label{eq-ener}
\begin{split}
& \frac{1}{2\Dt} \left(  \rnorm{\rhonp}^2-\rnorm{\rhon}^2+\rnorm{\rhonp-\rhon}^2\right)
+ \sumi \rhonpi\Dnot\av{v\gnp_i}\Dx \\
& + \frac{\eps^2}{2\Dt} \left( \gnorm{\gnp}^2-\gnorm{\gn}^2 +
  \gnorm{\gnp-\gn}^2  \right) + \eps
\innerp{\gnp}{\left(v^+\Dm+v^-\Dp\right)\gn} \\
& \leq  -\tsigma\gnorm{\gnp}^2 + \sumi \av{v\Dnot\gnp_i} \rhoni\Dx.
\end{split}
\end{equation}

\bigskip
\noindent {\it Step 2.} \\
In this step, we show how the $\rhonp-\rhon$ term can be eliminated
in~(\ref{eq-ener}). First, it is useful to write $\rhoni$ in the
right-hand side of~(\ref{eq-ener}) as $(\rhoni-\rhonpi)+\rhonpi$:
indeed, the terms $\sumi \rhonpi\Dnot\av{v\gnp_i}$ and $\sumi
\av{v\Dnot\gnp_i} \rhonpi$ in the left and right-hand sides cancel
out and we obtain:
\begin{equation}  \label{eq-enerbis}
\begin{split}
& \frac{1}{2\Dt} \left(  \rnorm{\rhonp}^2-\rnorm{\rhon}^2+\rnorm{\rhonp-\rhon}^2\right) \\
& + \frac{\eps^2}{2\Dt} \left( \gnorm{\gnp}^2-\gnorm{\gn}^2 +
  \gnorm{\gnp-\gn}^2  \right) + \eps²
\innerp{\gnp}{\left(v^+\Dm+v^-\Dp\right)\gn} \\
& =  -\tsigma\gnorm{\gnp}^2 + \sumi \av{v\Dnot\gnp_i} (\rhoni-\rhonpi)\Dx.²
\end{split}
\end{equation}
Now, we use
the following Young inequality:
\begin{equation}  \label{eq-Yrho}
\sumi \av{v\Dnot\gnp_i}(\rhoni-\rhonpi)\Dx \leq \alpha \rnorm{\rhonp-\rhon}^2+\frac{1}{4\alpha}\sumi\av{v\Dnot\gnp_i}^2\Dx.
\end{equation}
Then the $\rhonp-\rhon$ terms cancel out in~(\ref{eq-enerbis}) if $\alpha=\frac{1}{2\Dt}$
and we get 
\begin{equation}  \label{eq-fener}
\begin{split}
& \frac{1}{2\Dt} \left(  \rnorm{\rhonp}^2-\rnorm{\rhon}^2\right)
+ \frac{\eps^2}{2\Dt} \left( \gnorm{\gnp}^2-\gnorm{\gn}^2 +  \gnorm{\gnp-\gn}^2  \right) \\
& + \eps \innerp{\gnp}{\left(v^+\Dm+v^-\Dp\right)\gn} \leq  
-\tsigma\gnorm{\gnp}^2 +\frac{\Dt}{2}\sumi\av{v\Dnot\gnpi}^2\Dx.
\end{split}
\end{equation}

\bigskip
\noindent {\it Step 3.} \\
Here, we work on the inner product of~(\ref{eq-fener}) to show that
the $\gnp-\gn$ terms can also be eliminated. First, we insert $\gnp$
in this inner product to get:
\begin{equation}  \label{eq-vDgg1}
\begin{split}
& \innerp{\gnp}{\left(v^+\Dm+v^-\Dp\right)\gn}  \\
& = \innerp{\gnp}{\left(v^+\Dm+v^-\Dp\right)\gnp} + \innerp{\gnp}{\left(v^+\Dm+v^-\Dp\right)(\gn-\gnp)} \\
& = A+B,
\end{split}
\end{equation}
and we rewrite terms $A$ and $B$ as follows. For $A$, we use the
centered form of the upwind operator (lemma~\ref{lemma:5}) and
the discrete integration by parts of lemma~\ref{lemma:intbp} to get:
\begin{equation}\label{eq-vDgg2} 
\begin{split}
A & = \innerp{\gnp}{v\Dc\gnp} -\frac{\Dx}{2} \innerp{\gnp}{|v|\Dm\Dp\gnp} \\
&  = \frac{\Dx}{2} \innerp{\Dp\gnp}{|v|\Dp\gnp} \\
& = \frac{\Dx}{2} \sumi \av{|v|\left(\Dp\gnpi\right)^2}\Dx.
\end{split}
\end{equation}
For $B$, we also use the discrete integration by parts of
lemma~\ref{lemma:intbp} to get:
\begin{equation}\label{eq-vDgg3} 
\begin{split}
B= -\innerp{\left(v^+\Dp+v^-\Dm\right)\gnp}{\gn-\gnp}.
\end{split}
\end{equation}
Then we apply the inequality of lemma~\ref{lemma:3} to $B$ to get 
\begin{equation}  \label{eq-YB}
|B|\leq \alpha \gnorm{\gnp-\gn}^2+\frac{1}{4\alpha}\gnorm{|v|\Dp\gnp}^2.
\end{equation}
Therefore, using~(\ref{eq-fener}),~(\ref{eq-vDgg1}),~(\ref{eq-vDgg2})
and~(\ref{eq-YB}), we see that the $\gnp-\gn$ terms cancel out
in~(\ref{eq-fener}) if $\alpha=\frac{\eps}{2\Dt}$, and we get
\begin{equation}\label{eq-enera} 
\begin{split}
& \frac{1}{2\Dt} \left(  \rnorm{\rhonp}^2-\rnorm{\rhon}^2\right)
+ \frac{\eps^2}{2\Dt} \left( \gnorm{\gnp}^2-\gnorm{\gn}^2 \right) \\
& + \eps\frac{\Dx}{2} \sumi \av{|v|\left(\Dp\gnpi\right)^2}\Dx 
  - \frac{\Dt}{2}\gnorm{|v|\Dp\gnp}^2   \\
& \leq   - \tsigma\gnorm{\gnp}^2 +\frac{\Dt}{2}\sumi\av{v\Dnot\gnpi}^2\Dx.
\end{split}
\end{equation}

\bigskip
\noindent {\it Step 4.} \\
Now, we show how all the $\Dp \gnp$ and the $\Dnot \gnp$ terms can be
controlled by $\gnorm{\gnp}$. First, note that the term
$\frac{\Dt}{2}\gnorm{|v|\Dp\gnp}^2$ of the left-hand side
of~(\ref{eq-enera}) can be estimated as follows:
\begin{equation}\label{eq-majDg} 
\begin{split}
\frac{\Dt}{2}\gnorm{|v|\Dp\gnp}^2 & =  \frac{\Dt}{2} \sumi \av{|v|^2\left(\Dp\gnpi\right)^2}\Dx \\
& \leq \frac{\Dt}{2} \sumi \av{|v|\left(\Dp\gnpi\right)^2}\Dx,
\end{split}
\end{equation}
since $|v|\leq 1$. Moreover, using lemma~\ref{lemma:2} and a change
of indices shows that the
last term of the right-hand side of~(\ref{eq-enera}) satisfies
\begin{equation}  \label{eq-majDnotg}
\frac{\Dt}{2}\sumi\av{v\Dnot\gnpi}^2\Dx \leq \frac{\Dt}{4}\sumi\av{|v|\left(\Dp\gnpi\right)^2}\Dx.
\end{equation}
Finally, we use these two estimates in~(\ref{eq-enera}) to obtain:
\begin{equation}\label{eq-enerf} 
\begin{split}
& \frac{1}{2\Dt} \left(  \rnorm{\rhonp}^2-\rnorm{\rhon}^2\right)
+ \frac{\eps^2}{2\Dt} \left( \gnorm{\gnp}^2-\gnorm{\gn}^2 \right)   \\
& \leq  - \tsigma\gnorm{\gnp}^2 + \left(\frac{3\Dt}{4}- \eps\frac{\Dx}{2}\right) \sumi \av{|v|\left(\Dp\gnpi\right)^2}\Dx.
\end{split}
\end{equation}
Now, taking the positive part of the factor $(\frac{3\Dt}{4}- \eps\frac{\Dx}{2})$ of the
right-hand side of~(\ref{eq-enerf}), we have the estimate
\begin{equation}\label{eq-enerfb} 
\begin{split}
  \left(\frac{3\Dt}{4}- \eps\frac{\Dx}{2}\right) \sumi \av{|v|\left(\Dp\gnpi\right)^2}\Dx 
& \leq   \left(\frac{3\Dt}{4}- \eps\frac{\Dx}{2}\right)^+ \sumi \av{\left(\Dp\gnpi\right)^2}\Dx \\
& \leq   \left(\frac{3\Dt}{4}- \eps\frac{\Dx}{2}\right)^+ \frac{4}{\Dx^2} \gnorm{\gnp}^2,
\end{split}
\end{equation}
where we have used $|v|\leq 1$ and the estimate of lemma~\ref{lemma:estimDp}.

\bigskip
\noindent {\it Step 5.}  \\
Finally, estimates~(\ref{eq-enerf}) and~(\ref{eq-enerfb}) show that
\begin{equation*}
   \frac{1}{2\Dt} \left(  \rnorm{\rhonp}^2-\rnorm{\rhon}^2\right)
+ \frac{\eps^2}{2\Dt} \left( \gnorm{\gnp}^2-\gnorm{\gn}^2 \right) 
\leq \left( \left(\frac{3\Dt}{4}- \eps\frac{\Dx}{2}\right)^+ \frac{4}{\Dx^2}-\tsigma\right) \gnorm{\gnp}^2.
\end{equation*}
This means that we have the final energy estimate 
\begin{equation*}
  \rnorm{\rhonp}^2+ \eps^2\gnorm{\gnp}^2 \leq \rnorm{\rhon}^2+ \eps^2\gnorm{\gn}^2
\end{equation*}
if $\Dt$ is such that
\begin{equation*}
  \left(\frac{3\Dt}{4}- \eps\frac{\Dx}{2}\right)^+ \frac{4}{\Dx^2}\leq\tsigma.
\end{equation*}
Since $\tsigma\geq 0$, an equivalent condition is $(\frac{3\Dt}{4}- \eps\frac{\Dx}{2})
\frac{4}{\Dx^2}\leq\sigma$, which gives the sufficient condition
\begin{equation*}
  \Dt\leq \frac{\Dx^2\tsigma}{3} + \frac{2}{3}\eps\Dx, 
\end{equation*}
which proves the theorem.

\begin{remark} \label{rem:absorption}
As explained at the beginning of this section, when the absorption
term is non zero and is discretized implicitly, its
contribution $-\sumi \sigma_{A,i}(\rhonpi)^2\Dx$ to the energy estimate
is non-positive and plays no role in the previous analysis. However,
if we use instead an explicit discretization, then our analysis has to be modified. The
difference now is that there is the additional term
$-\sumi\sigma_{A,i}\rhonpi\rhoni\Dx$ in the right-hand side
of~(\ref{eq-ener}). The idea of step 2 can be applied to this term
(replace $\rhon$ by $(\rhon-\rhonp)+\rhonp$ and use a Young inequality) so that the
$\rhonp-\rhon$ terms cancel out, but now with
$\alpha=\frac{1}{2\Dt(1+\sigma_{A,M})}$. The other steps are the same,
except that some coefficients are different. In order to shorten the
paper, these details are left to the reader.
\end{remark}

\section{Error estimates}
\label{sec:acc}

In this section, we simplify the presentation by taking constant total
cross section $\sigma$  and kernel $s\equiv\half$, no absorption
($\sigma_A\equiv 0$), and no source term ($S\equiv 0$). These
assumptions are not restrictive at all, and our analysis could be
directly applied in the general case.

Let $T>0$ be some finite time. For this study, we assume that the exact solution $(\rho,g)$ of~(\ref{eq-mM}) has the
following regularity: 
\begin{equation}\label{eq-reg} 
\begin{split}
& \rho\in C^2([0,T],H^1(\R)) \cap C^0([0,T],H^3(\R)) \\
& g\in C^2([0,T],L^2([-1,1],H^1(\R))) \cap C^0([0,T],L^2([-1,1],H^3(\R))).
\end{split}
\end{equation}
Note that this assumption implies that $g$ is uniformly bounded with
respect to $\eps$, which is quite strong. In particular, this excludes
the case of initial or transition layers. Indeed, if, for instance,
$f$ is not isotropic at $t=0$, then $\rho(t=0)\neq f(t=0)$, and hence
$g(t=0)=\frac{1}{\eps}(f-\rho)|_{t=0}=O(\eps^{-1})$ cannot be uniformly bounded
with respect to $\eps$.

With this assumption, we can obtain the following result.
\begin{theo} \label{theo:error}
If $\Dt$ satisfies the following condition
\begin{equation}  \label{eq-eCFL}
\Dt\leq  \frac{\Dx^2\sigma}{6} + \frac{2}{3}\eps\Dx,
\end{equation}
then for every time $T>0$, there exists a constant $C$ independent of
$\Dt$, $\Dx$, and $\eps$, such
that the numerical solution obtained by scheme~(\ref{eq-scheme})
satisfies the following error estimate
\begin{equation*}
\begin{split}
& \max_{n, n\Dt\leq T}\left(\sumi
|\rho(t_n,x_i)-\rhoni|\Dx
+\eps\sumi\av{|g(t_n,x_{i+\demi},v)-\gni(v)|}\Dx \right) \\
& \qquad \leq C\left((1+\eps^2)\Dt+\Dx^2+\eps\Dx\right).
\end{split}
\end{equation*}

\end{theo}

This theorem is proved in the following sections: in
section~\ref{subsec:trunc}, we first derive the truncation error of
the scheme, then in section~\ref{subsec:error}, we apply the same
analysis as for the stability result to prove the theorem.

   \subsection{Truncation error}
   \label{subsec:trunc} 

Let $\ani$ and $\frac{1}{\eps^2}\bni$ be the truncation errors of
scheme~(\ref{eq-scheme}), that is to say the reminders obtained by
inserting the exact solution of~(\ref{eq-mM}) in
relations~(\ref{eq-scheme}):
 \begin{subequations} \label{eq-truncs} 
\begin{align}
& \frac{\rho(t_{n+1},x_i)-\rho(t_n,x_i)}{\Dt} +
\av{v\frac{g(t_{n+1},x_{i+\demi})-g(t_{n+1},x_{i-\demi})}{\Dx}} = \ani
 \label{eq-truncsrho}  \\
& \frac{g(t_{n+1},x_{i+\demi},.)-g(t_n,x_i,.)}{\Dt} \\
& + \frac{1}{\eps\Dx} (I-\av{.})
\left(v^+(g(t_n,x_{i+\demi},v)-g(t_n,x_{i-\demi},v)+v^-(g(t_n,x_{i+\frac{3}{2}},v)-g(t_n,x_{i+\demi},v))\right)
\label{eq-truncsg}  \\
& \qquad\qquad =
- \frac{\sigma}{\eps^2} g(t_{n+1},x_{i+\demi},v)
-\frac{1}{\eps^2}v\frac{\rho(t_n,x_{i+1})-\rho(t_n,x_i)}{\Dx} +\frac{1}{\eps^2}\bni. \nonumber
\end{align} 
 \end{subequations}
We can prove the
following estimate of these truncation errors:
\begin{lemma} \label{lemma:estimab}
There exists a constant $\tilde{C}$
  independent of $\Dt$, $\Dx$ and $\eps$, such that 
\begin{equation*}
  \rnorm{\an}+\gnorm{\bn} \leq \tilde{C}\left( (1+\eps^2)\Dt+\Dx^2+\eps\Dx \right)
\end{equation*}
for every $n$
\end{lemma}
  This lemma is proved by using standard simple techniques
  (Taylor-Lagrange formula of different orders). But to simplify the
  paper, the proof---which is a bit long---is given in
  appendix~\ref{app:trunc}.

Now, let $\trhoni$ and $\tgni$ be the convergence errors, that is to
say the sequences defined by
\begin{equation*}
  \trhoni=\rho(t_n,x_i)-\rhoni \quad \text{ and } \quad \tgni(v)=g(t_n,x_{i+\demi},v)-\gni(v).
\end{equation*}
Then these sequences satisfy the "perturbed" scheme, written in the following compact form:
\begin{subequations} \label{eq-tsch} 
\begin{align}
& \frac{\trhonpi-\trhoni}{\Dt} + \Dnot\av{v\tgnp_i} = \anpi \label{eq-tschrho}  \\
& \frac{\tgnpi-\tgni}{\Dt} + \frac{1}{\eps} (I-\av{.})
\left(v^+\Dm+v^-\Dp\right)\tgni = - \frac{\sigma}{\eps^2} \tgnpi 
- \frac{1}{\eps^2}v\,\dnot \trhon_{i+\half}+\frac{1}{\eps^2}\bni,\label{eq-tschg}
\end{align} 
\end{subequations}
with the homogeneous initial data $\tilde{\rho}^0_i=0$ and
$\tilde{g}^0_{i+\demi}=0$ for every $i$. 

   \subsection{Analysis of the convergence error}
   \label{subsec:error}
In this section, we apply the same
analysis as for the stability result to prove that $\trhon$ and $\tgn$
can be controlled by the truncation errors.

\noindent {\it Step 1.} \\
We multiply~(\ref{eq-tschrho}) by $\trhonpi$ and take the sum over $i$
to get
\begin{subequations}
\begin{equation}  \label{eq-tErho}
\frac{1}{2\Dt} \left(  \rnorm{\trhonp}^2-\rnorm{\trhon}^2+\rnorm{\trhonp-\trhon}^2\right)
+ \sumi \trhonpi\Dnot\av{v\tgnp_i}\Dx = \sumi \trhonpi\anpi\Dx.
\end{equation}
Second, we multiply~(\ref{eq-tschg}) by $\tgnpi$, we take the velocity
average, we sum over $i$, which yields
\begin{equation}  \label{eq-tEg}
\begin{split}
& \frac{1}{2\Dt} \left( \gnorm{\tgnp}^2-\gnorm{\tgn}^2 + \gnorm{\tgnp-\tgn}^2  \right)
+ \frac{1}{\eps}\innerp{\tgnp}{(I-\av{.})\left(v^+\Dm+v^-\Dp\right)\tgn} \\
&  = - \frac{\sigma}{\eps^2} \gnorm{\tgnp}^2 - \frac{1}{\eps^2}\sumi
\av{v\tgnpi}\dnot \trhon_{i+\half}\Dx
+\frac{1}{\eps^2}\innerp{\tgnp}{\bn}.
\end{split}
\end{equation}
\end{subequations}
Finally, we add up~(\ref{eq-tErho}) and $\eps^2$ times~(\ref{eq-tEg}), we
use~(\ref{eq-Innp}) and lemma~\ref{lemma:intbp} to get
\begin{equation}  \label{eq-tener1}
\begin{split}
& \frac{1}{2\Dt} \left(
  \rnorm{\trhonp}^2-\rnorm{\trhon}^2+\rnorm{\trhonp-\trhon}^2\right) 
+\sumi \trhonpi\Dnot\av{v\tgnp_i}\Dx \\
& + \frac{\eps^2}{2\Dt} \left( \gnorm{\tgnp}^2-\gnorm{\tgn}^2 +
  \gnorm{\tgnp-\tgn}^2  \right)  + \eps
\innerp{\tgnp}{\left(v^+\Dm+v^-\Dp\right)\tgn} \\
& =  
\sumi \trhoni\anpi\Dx - \sigma\gnorm{\tgnp}^2 + \sumi
\av{v\Dnot\tgnpi}\trhon_{i+\half}\Dx + \innerp{\tgnp}{\bn}.
\end{split}
\end{equation}

\bigskip
\noindent {\it Step 2.} \\
Here, we can copy the steps 2 to 4 of the stability analysis
(see section~\ref{sec:stab}) for relation~(\ref{eq-tener1}). Therefore,
skipping the details, we just give the resulting energy estimate:
\begin{equation}\label{eq-tener} 
\begin{split}
& \frac{1}{2\Dt} \left(  \rnorm{\trhonp}^2-\rnorm{\trhon}^2\right)
+ \frac{\eps^2}{2\Dt} \left( \gnorm{\tgnp}^2-\gnorm{\tgn}^2 \right)  \\
& \leq   - \sigma\gnorm{\tgnp}^2  + \left(\frac{3\Dt}{4}- \eps\frac{\Dx}{2}\right) \sumi
\av{|v|\left(\Dp\tgnpi\right)^2}\Dx \\
& \quad + \sumi \trhoni\anpi\Dx  + \innerp{\tgnp}{\bn}.
\end{split}
\end{equation}

\bigskip
\noindent {\it Step 3.} \\
We estimate the scalar products of the right-hand side
of~(\ref{eq-tener}) by using two
different Young inequalities:
\begin{equation}\label{eq-youngab} 
\begin{split}
& \sumi\trhoni\anpi \Dx \leq\demi\left(\rnorm{\trhonp}^2+\rnorm{\anp}^2\right) \\
& \innerp{\tgnp}{\bn} \leq \frac{\sigma}{2}\gnorm{\tgnp}^2 + \frac{1}{2\sigma}\rnorm{\bn}^2.
\end{split}
\end{equation}
Moreover, by taking the positive part of the factor $(\frac{3\Dt}{4}-
\eps\frac{\Dx}{2})$ in~(\ref{eq-tener}), we have the estimate
\begin{equation}  \label{eq-tenerfb}
\left(\frac{3\Dt}{4}-\eps\frac{\Dx}{2}\right)\sumi\av{|v|\left(\Dp\tgnpi\right)^2}\Dx
\leq \left(\frac{3\Dt}{4}-\eps\frac{\Dx}{2}\right)^+\frac{4}{\Dx^2}\gnorm{\tgnp}^2.
\end{equation}
Now, we can use the inequalities~(\ref{eq-youngab})
and~(\ref{eq-tenerfb}) to obtain the energy estimate
\begin{equation}  \label{eq-tenerdiss}
\begin{split}
& \rnorm{\trhonp}^2+\eps^2\gnorm{\tgnp}^2 \\ 
& \quad \leq (1+\Dt)\rnorm{\trhon}^2+\eps^2\gnorm{\tgn}^2 
+\left(-\frac{\sigma}{2}+\frac{4}{\Dx^2}\left(\frac{3\Dt}{4}-\eps\frac{\Dx}{2}\right)^+\right)2\Dt\gnorm{\tgnp}^2
\\
& \qquad + \Dt\rnorm{\anp}^2 + \frac{\Dt}{\sigma}\gnorm{\bn}^2.
\end{split}
\end{equation}
Note that the factor of $\gnorm{\gnp}^2$ in the right-hand side
of~(\ref{eq-tenerdiss}) is non-positive if
$\frac{4}{\Dx^2}\left(\frac{3\Dt}{4}-\eps\frac{\Dx}{2}\right)^+$ $\leq
  \frac{\sigma}{2}$. Since $\sigma>0$, an equivalent condition is $\frac{4}{\Dx^2}\left(\frac{3\Dt}{4}-\eps\frac{\Dx}{2}\right)\leq
  \frac{\sigma}{2}$, which gives the sufficient condition
\begin{equation*}
  \Dt\leq \frac{\Dx^2\sigma}{6} + \frac{2}{3}\eps\Dx,
\end{equation*}
which is a bit more restrictive than the CFL
condition~(\ref{eq-CFL}). In that case, the energy
estimate~(\ref{eq-tenerdiss}) can be simplified in
\begin{equation*}
\begin{split}
& \rnorm{\trhonp}^2+\eps^2\gnorm{\tgnp}^2 \\ 
& \quad \leq (1+\Dt) \left(
  \rnorm{\trhon}^2+\eps^2\gnorm{\tgn}^2\right) + \Dt C_{\sigma}\left(\rnorm{\anp}^2 +\gnorm{\bn}^2\right),
\end{split}
\end{equation*}
where $C_{\sigma}=1+\frac{1}{\sigma}$. 

Now, by using a simple recursion, we obtain
\begin{equation*}
\begin{split}
& \rnorm{\trhon}^2+\eps^2\gnorm{\tgn}^2 \\ 
& \quad \leq (1+\Dt)^n \left(
  \rnorm{\tilde{\rho}^0}^2+\eps^2\gnorm{\tilde{g}^0}^2  \right) \\
& \quad + \Dt C_{\sigma}\sum_{k=1}^n\left( \rnorm{a^{n-k}}^2+\gnorm{b^{n-k}}^2 \right)(1+\Dt)^{k-1}.
\end{split}
\end{equation*}
If $n$ is such that $n\Dt\leq T$, we can use the classical inequality
$(1+\Dt)^n\leq e^{n\Dt}\leq e^T$. Moreover, we can use
lemma~\ref{lemma:estimab} and the fact that
$\tilde{\rho}^0=\tilde{g}^0=0$ to get
\begin{equation*}
\begin{split}
& \rnorm{\trhon}^2+\eps^2\gnorm{\tgn}^2 \\ 
& \quad \leq e^TC_{\sigma}T\tilde{C}\left( (1+\eps^2)\Dt+\Dx^2+\eps\Dx \right)^2,
\end{split}
\end{equation*}
and hence
\begin{equation*}
\begin{split}
& \rnorm{\trhon}+\eps\gnorm{\tgn} \\ 
& \quad \leq C\left( (1+\eps^2)\Dt+\Dx^2+\eps\Dx \right),
\end{split}
\end{equation*}
where $C=e^{\frac{T}{2}}\sqrt{C_{\sigma}T\tilde{C}}$ is independent of $\Dt$, $\Dx$,
and $\eps$. This concludes the proof of the theorem.

\section{Conclusion}
\label{sec:concl}

In this paper, we have proposed a very simple stability proof for the
recent AP scheme of~\cite{LM2007}. An explicit CFL condition that can
be used in a computational code has been found: it insures that the
scheme is stable and accurate, independently of $\eps$. This condition
gives a classical parabolic CFL condition $\Dt=O(\Dx^2)$ when $\eps$
goes to 0. Our proof uses very basic and simple arguments (energy
estimates and Young inequalities), and is valid for one-dimensional
linear equations with non-constant coefficients, continuous velocity
variable, in the whole space. Our technique applies to general linear
collision operators like operators of neutron transport or linear radiative
transfer. By the same technique, we have also proved uniform error
estimates. We mention that, in a work in preparation~\cite{LM_nonlin},
we are able to apply our method to a simple non-linear
problem coming from radiative transfer.

In the future, the analysis of this scheme for initial boundary-value
problems will be investigated. It would also be important to extend the
scheme and its analysis to 2 or 3 dimensional problems.

\bibliographystyle{plain}

\bibliography{biblio}

\appendix
\section{Estimate of the truncation errors: proof of lemma~\ref{lemma:estimab}}
\label{app:trunc}

First, we remind (with no proof) that standard Taylor expansions give the following estimates for some time finite
differences.
\begin{lemma} \label{lemma:tf}
  \begin{itemize}
  \item[(i)]
  If $\psi\in C^1([0,T])$, then 
\begin{equation*}
  |\psi(\tnp)-\psi(\tn)|\leq \Dt \max_{[0,T]}|\psi'|.
\end{equation*}
  \item[(ii)]
If $\psi\in C^2([0,T])$, then 
\begin{equation*}
  |\frac{\psi(\tnp)-\psi(\tn)}{\Dt}-\psi'(\tn)|\leq \Dt \max_{[0,T]}|\psi''|.
\end{equation*}
  \end{itemize}
\end{lemma}
Then we also have the following estimates for some space finite
differences.
\begin{lemma} \label{lemma:sf}
  \begin{itemize}
  \item[(i)] If $\phi\in H^1(\R)$ and $\Dx\leq 1$, then 
\begin{equation*}
  \sumi\phi(x_i)^2 \Dx \leq 2\|\phi\|^2_{H^1}.
\end{equation*}
  \item[(ii)] If $\phi\in H^2(\R)$, then 
\begin{equation*}
  \sumi|\frac{\phi(x_{i+1})-\phi(x_i)}{\Dx}-\phi'(x_i)|^2 \Dx \leq
  \frac{\Dx^2}{3}\|\phi''\|^2_{L^2}.
\end{equation*}
  \item[(iii)] If $\phi\in H^3(\R)$, then  
\begin{equation*}
  \sumi|\frac{\phi(x_{i+1})-\phi(x_i)}{\Dx}-\phi'(x_{i+\demi})|^2 \Dx
  \leq \frac{\Dx^4}{320}\|\phi'''\|^2_{L^2}.
\end{equation*}
  \end{itemize}
\end{lemma}
\begin{proof}
  For (i), we write the difference of the continuous and discrete
  $L^2$ norms of $\phi$ as
\begin{equation*}
\begin{split}
  & \int_\R\phi(x)^2\, dx - \sumi\phi(x_i)^2 \Dx  \\
&  =  \sumi\int_{\x}^{\xp}(\phi(x)^2-\phi(\x)^2)\, dx 
 = \sumi\int_{\x}^{\xp}\left(\int_{\x}^x\frac{d}{dy}(\phi(y)^2)\,
  dy\right)dx \\
& = \sumi\int_{\x}^{\xp}\left(\int_{\x}^x 2\phi'(y)\phi(y)\, dy\right)dx . 
\end{split}
\end{equation*}
Then, using a simple Young inequality, we get:
\begin{equation*}
\begin{split}
\sumi\phi(x_i)^2 \Dx & 
\leq \int_\R\phi(x)^2\, dx +\sumi\int_{\x}^{\xp}\left(\int_{\x}^x
  \left(\phi'(y)^2+\phi(y)^2\right)\, dy\right)dx  \\
& \leq \int_\R\phi(x)^2\, dx +\sumi\int_{\x}^{\xp}\left(\int_{\x}^{\xp}
  \left(\phi'(y)^2+\phi(y)^2\right)\, dy\right)dx  \\
& =
\|\phi\|^2_{L^2}+\Dx\left(\|\phi'\|^2_{L^2}+\|\phi\|^2_{L^2}\right) \\
& \leq (1+\Dx)\|\phi\|^2_{H^1}.
\end{split}
\end{equation*}

For (ii), we use the Taylor-Lagrange formula up to first order to get
\begin{equation*}
  \frac{\phi(x_{i+1})-\phi(x_i)}{\Dx}-\phi'(x_i)=-\frac{1}{\Dx}\int_{\x}^{\xp}(x-\xp)\phi''(x)\, dx.
\end{equation*}
Then a simple Cauchy-Schwarz inequality gives the following:
\begin{equation*}
  \left|\frac{\phi(x_{i+1})-\phi(x_i)}{\Dx}-\phi'(x_i)\right|^2\leq \frac{\Dx}{3}\int_{\x}^{\xp}\phi''(x)^2\, dx.
\end{equation*}
Finally, we get the result by multiplying by $\Dx$ and taking the sum
over $i\in\Z$.

For (iii), we use the Taylor-Lagrange formula up to second order to
get the following two relations:
\begin{equation*}
\begin{split}
&\phi(\xp)-\phi(\xpd)=\demi\phi'(\xpd)+\frac{\Dx^2}{8}\phi''(\xpd)+\int_{\xpd}^{\xp}\frac{(x-\xp)^2}{2\Dx}\phi'''(x)\,dx
,\\
&\phi(\x)-\phi(\xpd)=-\demi\phi'(\xpd)+\frac{\Dx^2}{8}\phi''(\xpd)-\int_{\x}^{\xpd}\frac{(x-\x)^2}{2\Dx}\phi'''(x)\,dx.
\end{split}
\end{equation*}
Then, we take the difference of these relations to obtain:
\begin{equation*}
\frac{\phi(x_{i+1})-\phi(x_i)}{\Dx}-\phi'(x_{i+\demi}) =
\int_{\xpd}^{\xp}\frac{(x-\xp)^2}{2\Dx}\phi'''(x)\,dx + 
\int_{\x}^{\xpd}\frac{(x-\x)^2}{2\Dx}\phi'''(x)\,dx.
\end{equation*}
Now, using a Young then a Cauchy-Schwarz inequalities gives
\begin{equation*}
\left|  \frac{\phi(x_{i+1})-\phi(x_i)}{\Dx}-\phi'(x_{i+\demi})
\right|^2
\leq \frac{\Dx^3}{320}\int_{\x}^{\xp}\phi'''(x)^2\, dx.
\end{equation*}
Again, the final result is obtained by multiplying by $\Dx$ and taking the sum
over $i\in\Z$.
\end{proof}

Now we study the truncation error $\ani$ which is defined
by~(\ref{eq-truncsrho}). Since $(\rho,g)$ is the exact solution, we
have:
\begin{equation*}
  \ani= \left(\frac{\rho(t_{n+1},x_i)-\rho(t_n,x_i)}{\Dt}-\dt \rho(t_n,x_i)\right) +
\av{v\left(\frac{g(t_{n+1},x_{i+\demi})-g(t_{n+1},x_{i-\demi})}{\Dx}-\dx g(t_{n+1},x_i)\right)}
\end{equation*}
Then we estimate the norm of $\an$ as follows: we use a Young inequality,
lemma~\ref{lemma:tf}, (iii) of lemma~\ref{lemma:sf}, and
lemma~\ref{lemma:2}, and we obtain
\begin{equation*}
\sumi|\ani|^2\Dx\leq 2 
\left( 
\sumi  \Dt^2\max_{[0,T]}\left|\partial_{tt}\rho(t,x_i)\right|^2\Dx
+\demi\av{\frac{\Dx^4}{320}\left\|\partial_{xxx}g(\tnp,.,v)\right\|^2_{L^2}}
\right).
\end{equation*}
By using (i) of lemma~\ref{lemma:sf}, the first term of the right-hand
side of the previous estimate can be controlled by a norm of $\rho$
(if $\Dx\leq 1$) and we get
\begin{equation*}
\sumi|\ani|^2\Dx\leq 2 
\left( 2\Dt^2\|\rho\|^2_{X}
+\frac{1}{4}\frac{\Dx^4}{320}\|g\|^2_{\cal Y}
\right), 
\end{equation*}
where $X=C^2([0,T],H^1(\R))$ and ${\cal  Y}=C^0([0,T],L^2([-1,1],H^3(\R)))$.
Consequently, it is clear that $\rnorm{\an}\leq
\tilde{C}_1(\Dt+\Dx^2)$, where $\tilde{C}_1$ depends only on the norms
of $\rho$ and $g$.

Now, we study the truncation error $\bni$ which is defined by~(\ref{eq-truncsg}).
Again, since $(\rho,g)$ is the exact solution, we
have:
\begin{equation*}
\begin{split}
\bni  = & \, \eps^2\left(\frac{g(t_{n+1},x_{i+\demi},v)-g(t_n,x_i,v)}{\Dt}-\dt g(t_n,x_i,v)\right) \\
& + \eps(I-\av{.})\left(
v^+\left(\frac{g(t_n,x_{i+\demi},v)-g(t_n,x_{i-\demi},v)}{Dx}-\dx
  g(t_n,x_{i-\demi},v)\right)
\right.\\
& \hspace{13ex}\left. +v^-\left(\frac{g(t_n,x_{i+\frac{3}{2}},v)-g(t_n,x_{i+\demi},v)}{Dx}-\dx g(t_n,x_{i+\demi},v)\right)
                  \right) \\
& + \sigma\left(g(t_{n+1},x_{i+\demi},v)-g(t_{n},x_{i+\demi},v)\right) \\
& +v\left(\frac{\rho(t_n,x_{i+1})-\rho(t_n,x_i)}{\Dx}-\dx
  \rho(t_n,x_{i+\demi})\right) 
\end{split}
\end{equation*}
The sequel is similar, though a but longer, to what we did for
$\ani$. By using again lemmas~\ref{lemma:tf} and~\ref{lemma:sf}, and some
standard inequalities (Young and Jensen), the reader can easily find
that the following estimate is satisfied:
\begin{equation*}
\sumi\av{|\bni|^2}\Dx \leq 4\left(
\eps^42\Dt^2\|g\|^2_{\cal X} 
+ \eps^22\frac{4}{3}\Dx^2\|g\|^2_{\cal Y} 
+\sigma^22\Dt^2\|g\|^2_{\cal X} 
+\frac{\Dx^4}{5\times 2^6}\|\rho\|^2_{Y} 
\right),
\end{equation*}
where ${\cal
  X}=C^2([0,T],L^2([-1,1],H^1(\R)))$ and $Y=C^0([0,T],H^3(\R))$.
Therefore, we have $\gnorm{\bn}\leq \tilde{C}_2\left(
  (1+\eps^2)\Dt+\eps\Dx+\Dx^2 \right)$, where $C_2$ depends only on
the norms of $\rho$ and $g$. Finally,
the estimates of $\rnorm{\an}$ and $\gnorm{\bn}$ allow us to complete
the proof.

\end{document}